\documentclass{amsart}

\usepackage{amsmath}
\usepackage{amsfonts}
\usepackage{amssymb,enumerate}
\usepackage{amsthm}
\usepackage[all]{xy}
\usepackage{hyperref}

\theoremstyle{plain}
\newtheorem{lem}{Lemma}[section]
\newtheorem{cor}[lem]{Corollary}
\newtheorem{prop}[lem]{Proposition}
\newtheorem{thm}[lem]{Theorem}

\newtheorem{intthm}{Theorem}

\theoremstyle{definition}
\newtheorem{defn}[lem]{Definition}
\newtheorem{ex}[lem]{Example}
\newtheorem{question}[lem]{Question}

\newtheorem{disc}[lem]{Remark}

\newtheorem{fact}[lem]{Fact}

\newcommand{\ann}{\operatorname{Ann}}

\newcommand{\len}{\operatorname{len}}

\newcommand{\ext}{\operatorname{Ext}}

\newcommand{\HH}{\operatorname{H}}
\newcommand{\Hom}{\operatorname{Hom}}	

\newcommand{\spec}{\operatorname{Spec}}

\newcommand{\im}{\operatorname{Im}}

\newcommand{\Ker}{\operatorname{Ker}}

\newcommand{\ideal}[1]{\mathfrak{#1}}
\newcommand{\m}{\ideal{m}}

\newcommand{\comp}[1]{\widehat{#1}}

\newcommand{\supp}{\operatorname{Supp}}

\newcommand{\into}{\hookrightarrow}

\renewcommand{\geq}{\geqslant}

\renewcommand{\ker}{\Ker}
\renewcommand{\hom}{\Hom}

\newcommand{\Ext}[4][R]{\operatorname{Ext}_{#1}^{#2}(#3,#4)}

\newcommand{\Otimes}[3][R]{#2\otimes_{#1}#3}
\renewcommand{\Hom}[3][R]{\operatorname{Hom}_{#1}(#2,#3)}	
	
\newcommand{\Tor}[4][R]{\operatorname{Tor}^{#1}_{#2}(#3,#4)}

\newcommand{\md}[1]{#1^{\vee}}

\newcommand{\mdd}[1]{#1^{\vee \vee}}

\newcommand{\mfs}{\mathfrak{S}_0(R)}
\newcommand{\mfq}{\mathfrak{Q}_0(R)}

\numberwithin{equation}{lem}

\begin{document}

\bibliographystyle{amsplain}

\author{Bethany Kubik}

\address{Bethany Kubik, Department of Mathematical Sciences,
601 Thayer Road \#222,
West Point, NY 10996
USA}

\email{bethany.kubik@usma.edu}

\urladdr{http://www.dean.usma.edu/departments/math/people/kubik/}


\title 
{Quasidualizing Modules}



\begin{abstract}

We introduce and study ``quasidualizing" modules.
An artinian $R$-module $T$ is \emph{quasidualizing}
if the homothety map $\comp R\rightarrow\hom TT$ is an isomorphism and 
$\Ext{i}{T}{T}=0$ for each integer $i>0$.
Quasidualizing modules are associated to semidualizing modules
via Matlis duality.  
We investigate the associations
via Matlis duality between subclasses of  the Auslander class and Bass class
and subclasses of derived $T$-reflexive modules.
\end{abstract}

\maketitle

\section*{Introduction}

 Let $R$ be a commutative local noetherian ring with
 maximal ideal $\m$ and residue field $k=R/\m$.  The $\m$-adic
 completion of $R$ is denoted $\comp R$, the injective hull of $k$
 is $E=E_R(k)$, and the Matlis duality functor is $(-)^\vee=\hom -E$.
 
The motivation for this work comes from
the study of semidualizing modules.
Semidualizing modules were first introduced by Vasconcelos~\cite{vasconcelos:dtmc}.
A finitely generated $R$-module $C$ is \emph{semidualizing}
if the homothety map $R\rightarrow\hom CC$ is an isomorphism and 
$\Ext{i}{C}{C}=0$ for each integer $i>0$.  For example, $R$ is always a
semidualizing $R$-module.
Therefore duality with respect to $R$ is a special case
of duality with respect to a semidualizing module,
as is duality with respect to a dualizing $R$-module
when $R$ has one.  One the other hand,
Matlis duality is not covered in this way.  The goal of this paper is to remedy this by
introducing and studying the ``quasidualizing" modules:
An artinian $R$-module $T$ is \emph{quasidualizing}
if the homothety map $\comp R\rightarrow\hom TT$ is an isomorphism and 
$\Ext{i}{T}{T}=0$ for each integer $i>0$; see Definition~\ref{defn101215j}.
  For example, $E$ is always a
quasidualizing module.

This paper is concerned 
with the properties of quasidualizing modules and how they
compare with the properties of semidualizing modules.  
For instance, the next result gives a direct link between quasidualizing
modules and semidualizing modules
via Matlis duality; see Theorem~\ref{prop110113}.

\begin{intthm}\label{intthm110215}
If $R$ is complete, then the set of isomorphism classes
of semidualizing $R$-modules is in bijection with the set of isomorphism
classes of quasidualizing $R$-modules
by Matlis duality.
\end{intthm}

Following the literature on semidualizing modules, we use
quasidualizing modules to define other classes of modules.
For instance, given an $R$-module $M$, we consider
the class $\mathcal{G}_M^{\text{full}}(R)$ of 
``derived $M$-reflexive $R$-modules" and
their subclasses $\mathcal{G}_M^{\text{noeth}}(R)$
and $\mathcal{G}_M^{\text{artin}}(R)$ of noetherian
modules and artinian modules respectively.  We also
consider subclasses of the Auslander class
$\mathcal{A}_M(R)$ and the Bass class
$\mathcal{B}_M(R)$.  
See Section~\ref{sec110223} for definitions.
Some relations between these classes
are listed in the next result which is proved in Section~\ref{sec110301}.

\begin{intthm}\label{intthm110223}
Assume $R$ is complete, and let $T$ be a quasidualizing $R$-module.
Then we have the following inverse equivalences and equalities
\begin{enumerate}[\rm(i)]
\item 
$\xymatrix{
\mathcal{B}_{\md{T}}^{\text{noeth}}(R)\ar@<1ex>[r]^-{\md{(-)}}
&\mathcal{G}_{T}^{\text{artin}}(R)
=\mathcal{A}_{\md{T}}^{\text{artin}}(R)\ar@<1ex>[l]^-{\md{(-)}}
}$;
\item
$\xymatrix{
\mathcal{B}_{\md{T}}^{\text{artin}}(R)\ar@<1ex>[r]^-{\md{(-)}}
&\mathcal{G}_{T}^{\text{noeth}}(R)
=\mathcal{A}_{\md{T}}^{\text{noeth}}(R)\ar@<1ex>[l]^-{\md{(-)}}
}$; 
\item 
$\xymatrix{
\mathcal{B}_{T}^{\text{artin}}(R)\ar@<1ex>[r]^-{\md{(-)}}
&\mathcal{G}_{\md{T}}^{\text{noeth}}(R)
=\mathcal{A}_{T}^{\text{noeth}}(R)\ar@<1ex>[l]^-{\md{(-)}}
}$; and
\item 
$\xymatrix{
\mathcal{B}_{T}^{\text{noeth}}(R)\ar@<1ex>[r]^-{\md{(-)}}
&\mathcal{G}_{\md{T}}^{\text{artin}}(R)
=\mathcal{A}_{T}^{\text{artin}}(R)\ar@<1ex>[l]^-{\md{(-)}}
}$.
\end{enumerate}
\end{intthm}

As a consequence of the previous result, we conclude that
the classes $\mathcal{G}_{\md{T}}^{\text{noeth}}(R)$
and $\mathcal{G}_T^{\text{artin}}(R)$ are substantially different.
For instance, as we observe next $\mathcal{G}_T^{\text{artin}}(R)$
satisfies the two-of-three condition, while 
the class $\mathcal{G}_{\md{T}}^{\text{noeth}}(R)$ does not; see Theorem~\ref{prop110125}.

\begin{intthm}\label{intthm110215}
Assume that $R$ is complete, and let $T$ be a quasidualizing
$R$-module.  Then $\mathcal{G}_T^{\text{artin}}(R)$ satisfies the two-of-three condition, that is,
given an exact sequence of $R$-module homomorphisms
$0\rightarrow L_1\rightarrow L_2\rightarrow L_3\rightarrow 0$
if any two of the modules are in $\mathcal{G}_T^{\text{artin}}(R)$, then
so is the third.
\end{intthm}

In Section~\ref{sec110223} we provide some definitions and background
material.   Section~\ref{sec110225} describes properties
related to quasidualizing modules,
and Section~\ref{sec110301} describes the relations
between the different classes of modules using Matlis duality
as well as Theorem~\ref{intthm110215}.

\section{Background material}\label{sec110223}

\begin{defn}\label{defn110301}
We say that an $R$-module $L$ is
\emph{Matlis reflexive} if the natural
bidualitiy map $\delta_L^E:L\rightarrow L^{\vee\vee}$,
given by $l\mapsto [\phi\mapsto \phi(l)]$. 
\end{defn}

\begin{fact}\label{fact110203}
Let $L$ be an $R$-module.  The natural biduality map
$\delta_L$ is injective; 
see~\cite[Theorem 18.6(i)]{matsumura:crt}.  If $L$ is Matlis reflexive,
then $\md{L}$ is Matlis reflexive.
\end{fact}

\begin{fact}\label{fact110120b}
Assume $R$ is complete and let $L$ be an $R$-module.  If $L$ is artinian,
then $\md{L}$ is noetherian. If $L$ is noetherian, then $\md{L}$ is
artinian.  Since $R$ is complete, both artinian modules and noetherian
modules are Matlis reflexive;
see~\cite[Theorem 18.6(v)]{matsumura:crt}.
\end{fact}

\begin{lem}\label{lem110120}
Let $L$ and $L'$ be $R$-modules such that
$L$ is Matlis reflexive.  Then for all $i\geq0$ we have the isomorphisms
$$\Ext{i} {L'}{L}\cong\Ext{i}{\md{L}}{\md{L'}} \text{ and }
\Ext{i}{L'}{\md{L}}\cong\Ext{i}{L}{\md{L'}}.$$
\end{lem}

\begin{proof}
For the first isomorphism, since $L$ is Matlis reflexive,
by definition the map $$\Ext{i}{L'}{\delta_L}:\Ext{i}{L'}{L}\rightarrow
\Ext{i}{L'}{\hom{\md{L}}{E}}$$ is an isomorphism.
A manifestation of Hom-tensor adjointness yields the following isomorphisms
$$\Ext{i}{L'}{\hom{\md{L}}{E}}\xrightarrow{\cong}
\Ext{i}{\Otimes{L'}{\md{L}}}{E}\xrightarrow{\cong}
\Ext{i}{\md{L}}{\md{L'}}.$$
The composition of these maps provides us with the isomorphism
$\Ext{i} {L'}{L}\cong\Ext{i}{\md{L}}{\md{L'}}$.

For the second isomorphism, the fact that $L$ is Matlis reflexive explains
the second step in the following sequence
$\Ext{i}{L'}{\md{L}}\cong\Ext{i}{L^{\vee\vee}}{\md{L'}}\cong\Ext{i}{L}{\md{L'}}$.
The first step follows from the first isomorphism since $\md{L}$ is Matlis reflexive.
\end{proof}

\begin{fact}\label{fact110120a}
Assume $R$ is complete and let $A$ and $A'$ be artinian $R$-modules.
Then $\hom{A}{A'}$ is noetherian.
This can be deduced using~\cite[Theorem 2.11]{kubik:hamm1}.
\end{fact}

\begin{fact}\label{fact110518}
Let $L$ be an $R$-module.  Then $L$ is artinian over $R$ if
and only if it is artinian over $\comp R$.  See~\cite[Lemma~1.14]{kubik:hamm1}.
\end{fact}

\begin{lem}\label{fact101215a}
Assume $R$ is artinian and let $L$ be an $R$-module.  Then the following
are equivalent
\begin{enumerate}[\rm(i)]
\item \label{fact101215a1} $L$ is noetherian over $R$;
\item \label{fact101215a2}  $L$ is finitely generated over $R$; and
\item  \label{fact101215a3} $L$ is artinian.
\end{enumerate}
\end{lem}

\begin{proof}
The equivalence ~\eqref{fact101215a1}$\iff$~\eqref{fact101215a2} is
standard; see~\cite[Propositions 6.2 and 6.5]{atiyah:ica}.

For the implication ~\eqref{fact101215a2}$\implies$~\eqref{fact101215a3},
assume that $L$ is finitely generated over $R$.  Then there exists an $n\in\mathbb N$
and a surjective map
$R^n\xrightarrow{\phi}L$ so that we have
$L\cong\im(\phi)\cong R^n/\ker(\phi)$.  
Since $R$ is artinian,
$R^n$ is artinian.  Thus $L$ is artinian because the quotient of an artinain
module is artinian; see~\cite[Proposition 6.3]{atiyah:ica}.

For the implication~\eqref{fact101215a3}$\implies$~\eqref{fact101215a1},
assume that $L$ is artinian.  Then there exists an $n\in\mathbb N$ such
that $L\into E^n$; see~\cite[Theorem~3.4.3]{enochs:rha}.
  Since $R$ is artinian,
   we have $\md{R}\cong E$
is noetherian over $\comp R$ by Fact~\ref{fact110120b}, where the isomorphism 
follows from~\cite[Theorem 18.6 (iv)]{matsumura:crt}.  Hence we have that $E^n$ is noetherian over $\comp R=R$ since $R$ is artinian.   Since any submodule
of a noetherian module is noetherian, we conclude that $L$ is noetherian over $R$;
see~\cite[Proposition 6.3]{atiyah:ica}.
\end{proof}

%

\begin{lem}\label{lem110207a}
Assume $R$ is complete and let $A$ be an artinian $R$-module.  Then there
exists an injective resolution $I$ of $A$ such that  for each $i\geq 0$
we have $I_i\cong E^{b_i}$
for some  $b_i\in\mathbb{N}$.
Furthermore, $\md{I}$ is a free resolution of $\md{A}$.
\end{lem}

\begin{proof}
Since $A$ is artinian, we have the map
$A\into E^{b_0}$ for some $b_0\geq 1$; see~\cite[Theorem~3.4.3]{enochs:rha}.
Because the finite direct sum of artinian modules is artinian,
$E^{b_0}$ is artinian and we have $E^{b_0}/A\into E^{b_1}$
for some $b_1\geq 0$.  Recursively we can construct
an injective resolution of $A$ such that for each $i\geq 0$
we have $I_i\cong E^{b_i}$ for some $b_i\in\mathbb{N}$.

Next we show that $\md{I}$ is a free
resolution of $\md{A}$.
The fact that  $I_i\cong E^{b_i}$ explains the first step
in the following sequence
$$I_i^{\vee}=\hom{I_i}{E}\cong\hom{E^{b_i}}{E}
\cong\hom{E}{E}^{b_i}\cong \comp R^{b_i}\cong R^{b_i}.
$$
The second step is standard.  The third step
is from~\cite[Theorem~18.6(iv)]{matsumura:crt},
and the last step follows from the assumption that $R$ is complete.
The desired conclusion follows from the fact that $(-)^\vee$ is exact.
\end{proof}

\begin{defn}\label{defn110125}
Let $L$, $L'$ and $L''$ be $R$-modules.
The \emph{Hom-evaluation} morphism
$$\theta_{LL'L''}:\Otimes{L}{\hom {L'}{L''}}\rightarrow\hom{\hom{L}{L'}}{L"}$$
is given by $a\otimes\phi\mapsto[\beta\mapsto
\phi(\beta(a))]$.
\end{defn}

\begin{fact}\label{fact110125}
The Hom-evaluation morphism $\theta_{LL'L''}$ is an isomorphism if the modules satisfy one
of the following conditions:
\begin{enumerate}[\rm(a)]
\item $L$ is finitely generated and $L''$ is injective; or
\item $L$ is finitely generated and projective.
\end{enumerate}
See~\cite[Lemma~1.6]{ishikawa:oimafm} and~\cite[Lemma~3.55]{rotman:iha2}.
\end{fact}

%

\begin{defn}\label{defn101215d}
An $R$-module $C$ is \emph{semidualizing} if it satisfies the following
\begin{enumerate}[\rm(i)]
\item \label{defn101215d1} $C$ is finitely generated;
\item \label{defn101215d2} the homothety morphism $\chi_C^R:R\rightarrow\hom CC$, defined
by $r\mapsto[c\mapsto rc]$, is an isomorphism; and
\item \label{defn101215d3}one has  $\Ext{i}{C}{C}=0$ for all $i>0$.
\end{enumerate}
\end{defn}

\begin{disc}\label{rmk111220}
Let $\mathfrak{S}_0(R)$ denote the set of isomorphism classes of semidualizing 
$R$-modules.
\end{disc}

\begin{ex}\label{ex110301}
The ring $R$ is always semidualizing.
\end{ex}

\begin{defn}\label{defn101215j}
An $R$-module $T$ is \emph{quasidualizing} if it satisfies the following
\begin{enumerate}[\rm(i)]
\item \label{defn101215j1} $T$ is artinian;
\item \label{defn101215j2} the homothety morphism $\chi_T^{\comp R}:\comp R\rightarrow\hom TT$, defined by $r\mapsto[t\mapsto rt]$,  is an isomorphism; and
\item \label{defn101215j3}one has   $\Ext{i}{T}{T}=0$ for all $i>0$.
\end{enumerate}
\end{defn}

\begin{disc}\label{rmk120111}
The homothety morphism $\chi_T^{\comp R}$ is well defined since $T$ is
artinian implying by Fact~\ref{fact110518}
 that $T$ is an $\comp R$-module.
\end{disc}

\begin{disc}\label{rmk101215e} 
Let $\mathfrak{Q}_0(R)$ denote the set of isomorphism classes
of quasidualizing modules.
\end{disc}

\begin{ex}\label{ex110301a}
The injective hull of the residue field $E$ is always quasidualizing.
See~\cite[Theorem~3.4.1]{enochs:rha} and~\cite[Theorem~18.6(iv)]{matsumura:crt}
for conditions (i) and (ii)
of Definition~\ref{defn101215j}.  Since $E$ is injective by definition,
we have $\Ext{i}{E}{E}=0$ for all $i>0$ satisfying the last condition.
\end{ex}

\begin{defn}\label{defn110207}
Let $M$  be an $R$-module.  Then an $R$-module $L$ is 
\emph{derived $M$-reflexive} if
\begin{enumerate}[\rm(i)]
\item \label{defn110207a} the natural biduality map $\delta_L^M: L\rightarrow\hom{\hom LM}{M}$
defined by $l\mapsto[\phi\mapsto\phi(l)]$ is an isomorphism; and
\item \label{defn110207b}one has $\Ext{i}{L}{M}=0=\Ext{i}{\hom LM}{M}$ for all $i>0$.
\end{enumerate}
We write $\mathcal{G}_M^{full}(R)$ to denote the class of all 
derived $M$-reflexive $R$-modules, $\mathcal{G}_M^{\text{mr}}(R)$
to denote the class of all Matlis reflexive derived $M$-reflexive $R$-modules,
 $\mathcal{G}_M^{\text{artin}}(R)$
to denote the class of all artinian derived $M$-reflexive $R$-modules, and
 $\mathcal{G}_M^{\text{noeth}}(R)$
to denote the class of all noetherian derived $M$-reflexive $R$-modules.
\end{defn}

\begin{disc}\label{rmk110207}
When $M=C$ is a semidualizing $R$-module, the class
$\mathcal{G}_M^{\text{noeth}}(R)$ is the class of \emph{totally $C$-reflexive}
$R$-modules, sometimes denoted $\mathcal{G}_C(R)$.
\end{disc}

\begin{defn}\label{defn110113}
Let $L$ and $L'$ be $R$-modules.  We say that $L$ is in the
\emph{Bass class}
$\mathcal B_{L'}(R)$ with respect to $L'$ if it satisfies the following:
\begin{enumerate}[\rm(i)]
\item the natural evaluation homomorphism
 $\xi_L^{L'}: \Otimes{\hom{L'}{L}}{L'}\rightarrow L$, defined by
$\phi\otimes l\mapsto\phi(l)$, is an isomorphism; and
\item one has $\Ext{i}{L'}{L}=0=\Tor{i}{L'}{\hom{L'}{L}}$ for all $i> 0$.
\end{enumerate}
We write $\mathcal{B}_{L'}^{\text{mr}}(R)$ to denote the class
of all Matlis reflexive $R$-modules in the Bass class with respect to $L'$.
We write $\mathcal{B}_{L'}^{\text{artin}}(R)$ to denote the class
of all artinian $R$-modules in the Bass class with respect to $L'$, and
$\mathcal{B}_{L'}^{\text{noeth}}(R)$ to denote the class
of all noetherian $R$-modules in the Bass class with respect to $L'$.
\end{defn}

\begin{defn}\label{defn110113a}
Let $L$ and $L'$ be $R$-modules.  We say that $L$ is in the
\emph{Auslander class} $\mathcal{A}_{L'}(R)$ with respect to $L'$
 if it satisfies the following
\begin{enumerate}[\rm(i)]
\item the natural homomorphism 
$\gamma_L^{L'}:L\rightarrow\hom{L'}{\Otimes{L'}{L}}$, which is defined by
$l\mapsto [l'\mapsto l'\otimes l]$, is an isomorphism; and
\item one has $\Tor{i}{L'}{L}=0=\Ext{i}{L'}{\Otimes{L'}{L}}$ for all $i> 0$.
\end{enumerate}
We write $\mathcal{A}_{L'}^{\text{mr}}(R)$ to denote the class
of all Matlis reflexive $R$-modules in the Auslander class with respect to $L'$.
We write $\mathcal{A}_{L'}^{\text{artin}}(R)$ to denote the class
of all artinian $R$-modules in the Auslander class with respect to $L'$, and
 $\mathcal{A}_{L'}^{\text{noeth}}(R)$ to denote the class
of all noetherian $R$-modules in the Auslander class with respect to $L'$.
\end{defn}

\section{Quasidualizing Modules}\label{sec110225}

We begin with a few preliminary results pertaining to quasidualizing
modules.

\begin{prop}\label{prop110127}
Let $T$ be an $R$-module.  Then $T$ is a quasidualizing $R$-module
if and only if $T$ is a quasidualizing $\comp R$-module.
\end{prop}

\begin{proof}
We need to check the equivalence of three conditions.  For the first condition,
$T$ is an artinian $R$-module if and only if $T$ is an artinian
$\comp R$-module by Fact~\ref{fact110518}.  For the rest
of the proof we assume without loss of generality that $T$ is artinian.

For the second condition, we have the equality $\hom TT=\hom[\comp R] TT$
from the fact that 
$T$ is $\m$-torsion and~\cite[Lemma 1.5(a)]{kubik:hamm1}.
  This explains the equality
in the following commutative diagram.
$$\xymatrix{
\comp R\ar[r]^-{\chi_T^{\comp R}}\ar[d]_-{\cong}
&\hom TT\ar[d]^-{=}\\
\comp{\comp R}\ar[r]^-{\chi_T^{\comp{\comp R}}}
&\hom[\comp R] TT
}$$
Since $\comp R\cong\comp{\comp R}$, we have
$\chi_T^{\comp R}$ is an isomorphism if and only if 
$\chi_T^{\comp {\comp R}}$ is an isomorphism.

For the last condition, 
Lemma~\ref{lem110207a} implies
that there exists  an injective resolution $I$ of $T$ such that
for each $i\geq 0$ we have $I_i\cong E^{b_i}$ for some $b_i\in\mathbb{N}$.
 For all $i\geq 0$, the modules $T$ and $I_i$ are artinian and hence
 $\m$-torsion.
%
By~\cite[Lemma 1.5(a)]{kubik:hamm1}, we have the equality
$\hom[\comp R]{T}{I_i}=\hom{T}{I_i}$ and $I$ is an injective
resolution of $T$ over $\comp {R}$.  This explains
the first and second steps in the next display:
$$\Ext[\comp R]{i}{T}{T}\cong\HH_{-i}(\hom[\comp R]{T}{I_i})\cong
\HH_{-i}(\hom{T}{I_i})\cong\Ext{i}{T}{T}.$$
The third step is by definition.
Thus we have $\Ext[\comp R]{i}{T}{T}=0$ for all $i>0$ if and only if
$\Ext{i}{T}{T}=0$ for all $i>0$.
\end{proof}

\begin{prop}\label{lem110204}
The following conditions are equivalent
\begin{enumerate}[\rm(i)]
\item \label{lem110204a}$E$ is a semidualizing $R$-module;
\item \label{lem110204d}$R$ is a quasidualizing $R$-module;
\item \label{lem110204f}$E$ is a noetherian $R$-module;
\item \label{lem110204b}$R$ is an artinian ring;
\item \label{lem110204c}$\mathfrak{Q}_0(R)=\mathfrak{S}_0(R)$; and
\item \label{lem110204e}$\mathfrak{Q}_0(R)\cap\mathfrak{S}_0(R)\neq 0$.
\end{enumerate}
\end{prop}

\begin{proof}
\eqref{lem110204f}$\iff$\eqref{lem110204b}
By~\cite[Theorem 18.6 (ii)]{matsumura:crt} we have $\len_R(R)=\len_R(\md{R})=\len_R(E)$,
where $\len_R(L)$ denotes the length of an $R$-module $L$.
Since $R$ is noetherian by assumption, we have $R$ is artinian if and
only if $R$ has finite length if and only if $\md{R}=E$ has finite length (by
the equalities above), if and only if $E$ is noetherian over $R$ (since $E$ is artinian; see~\cite[Theorem~3.4.1]{enochs:rha}).
That is, $R$ is artinian if and only if $E$ is noetherian over $R$.

\eqref{lem110204a}$\implies$\eqref{lem110204f}
If $E$ is a semidualizing $R$-module, then $E$
is noetherian  over $R$ by definition. 

\eqref{lem110204b}$\implies$\eqref{lem110204a}
Assume that $R$ is artinian.  Then $E$ is finitely generated
by the equivalence~\eqref{lem110204f}$\iff$~\eqref{lem110204b}.
  We have $R\cong\comp R$ since $R$ is artinian,
  and $\comp R\cong\hom EE$ by~\cite[Theorem 18.6 (iv)]{matsumura:crt}
  explaining
the unspecified isomorphisms in the following commutative diagram.
$$\xymatrix{
R\ar[r]^-{\chi_E^R}\ar[d]_-{\cong}
&\hom EE\\
\comp R\ar[ur]_-{\cong}
}$$
 Hence we conclude that the homothety morphism
$\chi_E^R$ is an isomorphism.
Since $E$ is injective, we have that $\Ext{i}{E}{E}=0$ for all $i>0$.  Thus $E$ is a
semidualizing $R$-module.

\eqref{lem110204b}$\implies$\eqref{lem110204c}  Assume that $R$ is artinian,
and let $L$ be an $R$-module.  We show that $L$ is a semidualizing
module if and only if $L$ is a quasidualizing module.  We need
to check the equivalence of three conditions.
For the first condition, $L$ is finitely generated if and only if $L$ is
artinian by Lemma~\ref{fact101215a}.  For the second condition,
the fact that $R$ is artinian implies that $\comp R\cong R$.
This explains the unlabeled isomorphism in the following commutative
diagram
$$\xymatrix{
R\ar[d]_-{\chi_L^R}\ar[r]^{\cong}
&\comp R\ar[dl]^-{\chi_L^{\comp R}}\\
\hom LL.
}$$
Thus the map $\chi_L^R$ is an isomorphism if and only if
the map $\chi_L^{\comp R}$ is an isomorphism.   The $\ext$ vanishing conditions
are equivalent by definition.

For the implication
\eqref{lem110204c}$\implies$\eqref{lem110204d}, assume
that $\mfq=\mfs$.  The $R$-module $R$ is always semidualizing .
Then by assumption it is also a quasidualizing $R$-module.

The implication \eqref{lem110204d}$\implies$\eqref{lem110204b} is evident
since $R$ is an artinian ring if and only if it is an artinian $R$-module.
For the implication \eqref{lem110204d}$\implies$\eqref{lem110204e},
if $R$ is a quasidualizing $R$-module, then the intersection
$\mfq\cap\mfs$ is nonempty since $R$ is also a semidualizing $R$-module.

For the implication \eqref{lem110204e}$\implies$\eqref{lem110204d},
assume that the intersection $\mfq\cap\mfs$ is nonempty.
Let $L\in\mfq\cap\mfs$.  Then $L$ is artinian
and noetherian, so it has finite length.  Since $L$ is artinian,
it is $\m$-torsion and by~\cite[Fact 1.2(b)]{kubik:hamm1} we have
$\supp_R(L)\subseteq\{\m\}$.  Since $L$ is a semidualizing $R$-module,
the map $R\rightarrow \hom LL$ is an isomorphism
so we have $\ann_R(L)\subseteq\ann_R(R)=0$.  This explains the
second step in the following sequence
$$\supp_R(L)=V(\ann_R(L))=V(0)=\spec(R).$$
  Thus
$\spec(R)=\supp_R(L)\subseteq\{\m\}\subseteq\spec(R)$ and
we conclude that
$\spec(R)=\{\m\}$.  Thus~\cite[Theorem 8.5]{atiyah:ica}
implies that $R$ is artinian. 
\end{proof}

\section{Classes of Modules and Matlis Duality}\label{sec110301}

This section explores the connections between the class of
quasidualizing $R$-modules and the class
of semidualizing $R$-modules as well as connections
between different subclasses of $\mathcal{A}_M(R)$, $\mathcal{B}_M(R)$,
and $\mathcal{G}_M^{full}(R)$.
The instrument used to detect
these connections is Matlis Duality.

\begin{thm}\label{prop110113}
Assume that $R$ is complete.  Then the maps
$\xymatrix{
\mathfrak{S}_{0}(R)\ar@<1ex>[r]^-{\md{(-)}}
&\mathfrak{Q}_{0}(R)\ar@<1ex>[l]^-{\md{(-)}}
}$
are inverse bijections.
\end{thm}

\begin{proof}
Let $C\in\mathfrak{S}_{0}(R)$.
We show that $\md{C}\in\mathfrak{Q}_{0}(R)$.   Fact~\ref{fact110120b} implies
that $\md{C}$ is artinian.
In the following commutative diagram,
the unspecified isomorphisms are from Hom-tensor
adjointness and the commutativity of tensor product
$$\xymatrix{
R\ar[r]^-{\chi^R_{\md{C}}}\ar[d]^-{\chi^R_C}
&\hom{\md{C}}{\hom CE}\ar[d]^-{\cong}\\
\hom CC\ar[d]^-{\hom{C}{\delta_C^E}}_-{\cong}
&\hom{\Otimes{\md{C}}{C}}{E}\ar[d]^-{\cong}\\
\hom{C}{\hom{\md{C}}{E}}\ar[r]^-{\cong}
&\hom{\Otimes{C}{\md{C}}}{E}.
}$$
Since $C\in \mfs$, it follows that $\chi^R_C$ is an isomorphism.
Fact~\ref{fact110120b} implies that the map $\delta_C^E$, and by
extension the map $\hom{C}{\delta_C^E}$, is an
isomorphism.  
 Hence we conclude from the diagram that $\chi^R_{\md{C}}$ is an
isomorphism.

For the last condition, Lemma~\ref{lem110120} explains the
first step in the following sequence 
$$\Ext{i}{\md{C}}{\md{C}}\cong
\Ext{i}{C}{C}=0.$$
  The second step follows from the fact
that $C$ is a semidualizing module.
Thus $\md{C}$ is a quasidualizing module. 

A similar argument shows that given a quasidualizing
 $R$-module $T$, 
the module $\md{T}$ is semidualizing.
Fact~\ref{fact110120b} implies
that $C\cong C^{\vee\vee}$ and $T\cong T^{\vee\vee}$, 
so that the given maps $\mfs\xrightarrow{\md{(-)}}\mfq$ and
$\mfq\xrightarrow{\md{(-)}}\mfs$
are inverse equivalences.
\end{proof}

\begin{ex}\label{ex110301b}
Assume that $R$ is Cohen-Macaulay and complete and admits a dualizing
module $D$.  The fact that $D$ is dualizing means
that $D$ is semidualizing and has finite injective dimension.
Therefore, by Theorem~\ref{prop110113}, we conclude
that $\md{D}$ is quasidualizing.
\end{ex}

\begin{prop}\label{prop110225}
Assume that  $R$ is complete and let $T$ be a quasidualizing $R$-module.  Then the maps
$\xymatrix{
\mathcal{B}_{\md{T}}^{\text{mr}}(R)\ar@<1ex>[r]^-{\md{(-)}}
&\mathcal{G}_{T}^{\text{mr}}(R)\ar@<1ex>[l]^-{\md{(-)}}
}$
are inverse bijections.
\end{prop}

\begin{proof}
Let $M$ be a Matlis reflexive $R$-module.  We show that if
$M\in\mathcal{B}_{\md{T}}^{\text{mr}}(R)$ then
$\md{M}\in\mathcal{G}_T^{\text{mr}}(R)$.  
Fact~\ref{fact110203} implies that $\md{M}$ is Matlis reflexive.
There are three remaining conditions to check. 

First we show that $\Ext{i}{\md{M}}{T}=0$ for all $i>0$.
Since $T$ is artinian and $R$ is complete, Fact~\ref{fact110120b}
implies that $T$ is Matlis reflexive, so we have
\begin{equation}\label{eqprop110224}
\Ext{i}{\md{M}}{T}\cong
\Ext{i}{\md{T}}{M}.
\end{equation}
by Lemma~\ref{lem110120}.
We have $\Ext{i}{\md{T}}{M}=0$ for all $i> 0$
since $M\in\mathcal{B}_{\md{T}}^{\text{mr}}(R)$.  Thus we conclude
$\Ext{i}{\md{M}}{T}=0$ for all $i> 0$.

Next we show that the map $\delta_{\md{M}}^{T}$ is an isomorphism.
The fact that $M\in\mathcal{B}_{\md{T}}^{\text{mr}}(R)$
implies the map $\xi_M^{\md{T}}$ is an isomorphism.  Therefore the map
$\hom{\xi_M^{\md{T}}}{E}$ in the following commutative diagram
is an isomorphism
$$\xymatrix@C=7em{
\md{M}\ar[r]_-{\cong}^-{\hom{\xi_M^{\md{T}}}{E}}\ar[d]_-{\delta_{\md{M}}^T}
&\hom{\Otimes{\hom{\md{T}}{M}}{\md{T}}}{E}\ar[d]^{\cong}\\
\hom{\hom {\md{M}}{T}}{T}\ar[r]_{\cong}
&\hom{\hom{\md{T}}{M}}{T}.
}$$
The unspecified isomorphisms are from Hom-tensor adjointness
and the isomorphism~\eqref{eqprop110224}.
Hence we conclude from the diagram that
$\delta_{\md{M}}^T$ is an isomorphism.

For the last condition, let $I$ be an injective resolution 
of $T$ such that for each $i\geq 0$ we have
$I_i\cong E^{b_i}$ for some $b_i\in\mathbb N$.
Lemma~\ref{lem110207a} implies that $\md{I}$
is a free resolution of $\md{T}$.
This explains steps 
(2) and 
(6)
in the following sequence
\begin{align*}
\Ext{i}{\hom{\md{M}}{T}}{T}&
\stackrel{(1)}{\cong}\Ext{i}{\hom{\md{T}}{M}}{T}\\
&\stackrel{(2)}{\cong}\HH_{-i}(\hom{\hom{\md{T}}{M}}{I})\\
&\stackrel{(3)}{\cong}\HH_{-i}(\hom{\hom{\md{T}}{M}}{I^{\vee\vee}})\\
&\stackrel{(4)}{\cong}\HH_{-i}(\hom{\Otimes{\hom{\md{T}}{M}}{\md{I}}}{E})\\
&\stackrel{(5)}{\cong}\hom{\HH_i(\Otimes{\md{I}}{\hom{\md{T}}{M}})}{E}\\
&\stackrel{(6)}{\cong}\hom{\Tor{i}{\md{T}}{\hom{\md{T}}{M}}}{E}.
\end{align*}
Step (1) follows from the isomorphism~\eqref{eqprop110224}.
Step (3) follows from the fact that any finite direct sum
of artinian modules is artinian; thus $I_j$ is artinian for
all $j$ and we
can apply Fact~\ref{fact110120b}.
Step (4) follows from Hom-tensor adjointness, and
step (5) follows from the fact that $E$ is injective and homology
commutes with exact functors.  
Since $M\in\mathcal{B}_{\md{T}}^{\text{mr}}(R)$, we have
$\Tor{i}{M}{\hom{\md{T}}{M}}=0$ for all $i>0$.  Hence
we conclude that 
$$\Ext{i}{\hom{\md{M}}{T}}{T}\cong
\hom{\Tor{i}{M}{\hom{\md{T}}{M}}}{E}=0$$ for all $i>0$.

Given an $R$-module $M'\in\mathcal{G}_T^{\text{mr}}(R)$, the argument
to show that 
$\md{M'}\in\mathcal{B}_{\md{T}}^{\text{mr}}(R)$ is similar.
Since $M$ and $M'$ are Matlis reflexive, that is $M\cong\mdd{M}$
and $M'\cong\mdd{M'}$, we conclude that the maps
$\mathcal{B}_{\md{T}}^{\text{mr}}(R)\xrightarrow{\md{(-)}}
\mathcal{G}_T^{\text{mr}}(R)$ and 
$\mathcal{G}_T^{\text{mr}}(R)\xrightarrow{\md{(-)}}\mathcal{B}_{\md{T}}^{\text{mr}}(R)$ are inverse equivalences.
\end{proof}

\begin{cor}\label{prop110224}
Assume that  $R$ is complete and let $T$ be a quasidualizing $R$-module.  Then the 
following maps are inverse bijections
$$\xymatrix{
\mathcal{B}_{\md{T}}^{\text{noeth}}(R)\ar@<1ex>[r]^-{\md{(-)}}
&\mathcal{G}_{T}^{\text{artin}}(R)\ar@<1ex>[l]^-{\md{(-)}}
&
\text{and}
&\mathcal{B}_{\md{T}}^{\text{artin}}(R)\ar@<1ex>[r]^-{\md{(-)}}
&\mathcal{G}_{T}^{\text{noeth}}(R)\ar@<1ex>[l]^-{\md{(-)}}.
}$$

\end{cor}

\begin{proof}
Fact~\ref{fact110120b} implies that if $N$ is a noetherian
$R$-module, then $\md{N}$ is an artinian $R$-module
and $N\cong N^{\vee\vee}$.  Furthermore,
if $A$ is an artinan $R$-module, then $\md{A}$ is
a noetherian $R$-module and $A\cong A^{\vee\vee}$.
Together with Proposition~\ref{prop110225}, this implies
that the maps 
$\xymatrix{
\mathcal{B}_{\md{T}}^{\text{noeth}}(R)\ar@<1ex>[r]^-{\md{(-)}}
&\mathcal{G}_{T}^{\text{artin}}(R)\ar@<1ex>[l]^-{\md{(-)}}
}$
are inverse bijections.
The proof for $\xymatrix{
\mathcal{B}_{\md{T}}^{\text{artin}}(R)\ar@<1ex>[r]^-{\md{(-)}}
&\mathcal{G}_{T}^{\text{noeth}}(R)\ar@<1ex>[l]^-{\md{(-)}}
}$ is similar.
\end{proof}

\begin{prop}\label{prop120120}
Assume that  $R$ is complete and let $T$ be a quasidualizing $R$-module.  Then the maps
$\xymatrix{
\mathcal{B}_{T}^{\text{mr}}(R)\ar@<1ex>[r]^-{\md{(-)}}
&\mathcal{G}_{\md{T}}^{\text{mr}}(R)\ar@<1ex>[l]^-{\md{(-)}}
}$
are inverse bijections.
\end{prop}

\begin{proof}
Let $M$ be a Matlis reflexive $R$-module.  
We show that if $M\in\mathcal{G}^{\text{mr}}_{\md{T}}(R)$, then
 $\md{M}\in\mathcal{B}^{\text{mr}}_T(R)$.  First we show that
the map $\xi_{\md{M}}^{T}$ is an isomorphism.  The fact that $M$
is Matlis reflexive implies that the map $\delta^E_M$ in the following
commutative diagram is an isomorphism
$$\xymatrix{
M\ar[r]^-{\delta_M^E}_-{\cong}\ar[d]^-{\delta_M^{\md{T}}}
&\mdd{M}\ar[d]^-{\md{(\xi_{\md{M}}^T)}}\\
\hom{\hom{M}{\md{T}}}{\md{T}}\ar[d]^-{\cong}
&\hom{\Otimes{\hom{T}{\md{M}}}{T}}{E}\\
\hom{\hom{T}{\md{M}}}{\md{T}}\ar[ur]_-{\cong}.
}$$
The unspecified 
isomorphisms are
from Hom-tensor adjointness and Lemma~\ref{lem110120}.
Since $M\in\mathcal{G}^{\text{mr}}_{\md{T}}(R)$, we have
that the map $\delta^{\md{T}}_M$ is an isomorphism.
Hence $\md{(\xi_{\md{M}}^T)}$ is an isomorphism.
Since $E$ is faithfully injective, this implies that
$\xi_{\md{M}}^T$ is an isomorphism.

Next we show that $\Ext{i}{T}{\md{M}}=0$ for all $i>0$.
Since $M$ is Matlis reflexive, Lemma~\ref{lem110120}
explains the first step in the following sequence
$\Ext{i}{T}{\md{M}}\cong\Ext{i}{M}{\md{T}}=0$.
The second step follows from the fact that
$M\in\mathcal{G}_{\md{T}}^{\text{mr}}(R)$.

Lastly, we show that $\Tor{i}{T}{\hom{T}{\md{M}}}=0$
for all $i>0$.  The commutativity of tensor
product explains the first step in the following
sequence
\begin{align*}
\md{\Tor{i}{T}{\hom{T}{\md{M}}}}&\cong
\md{\Tor{i}{\hom{T}{\md{M}}}{T}}\\
&\cong
\Ext{i}{\hom{T}{\md{M}}}{\md{T}}\\
&\cong
\Ext{i}{\hom{M}{\md{T}}}{\md{T}}\\
&=0.
\end{align*}
The second step follows from~\cite[Remark~1.9]{kubik:hamm1} and
the third step follows from Lemma~\ref{lem110120}.
The last step follows from the fact that $M\in\mathcal{G}_{\md{T}}^{\text{mr}}(R)$.

Given an $R$-module $M'\in\mathcal{B}_{T}^{\text{mr}}(R)$, the argument
to show that $\md{M'}\in\mathcal{G}_{\md{T}}^{\text{mr}}(R)$ is similar but easier.
Since $M$ and $M'$ are Matlis reflexive, we conclude that
the maps $\mathcal{B}_T^{\text{mr}}(R)\xrightarrow{\md{(-)}}\mathcal{G}_{\md{T}}^{\text{mr}}(R)$
and $\mathcal{G}_{\md{T}}^{\text{mr}}(R)\xrightarrow{\md{(-)}}\mathcal{B}_T^{\text{mr}}(R)$
are inverse equivalences.
\end{proof}

\begin{cor}\label{prop110224}
Assume that  $R$ is complete and let $T$ be a quasidualizing $R$-module.  Then the 
following maps are inverse bijections
$$\xymatrix{
\mathcal{B}_{T}^{\text{noeth}}(R)\ar@<1ex>[r]^-{\md{(-)}}
&\mathcal{G}_{\md{T}}^{\text{artin}}(R)\ar@<1ex>[l]^-{\md{(-)}}
&
\text{and}
&\mathcal{B}_{T}^{\text{artin}}(R)\ar@<1ex>[r]^-{\md{(-)}}
&\mathcal{G}_{\md{T}}^{\text{noeth}}(R)\ar@<1ex>[l]^-{\md{(-)}}.
}$$
\end{cor}

The next proposition establishes the relationship between
a subclass of the Auslander class and a subclass of the
derived reflexive modules.

\begin{prop}\label{prop110130}
If $R$ is complete and $T$ is a quasidualizing $R$-module,
then 
$$\mathcal{G}_{\md{T}}^{\text{mr}}(R)=\mathcal{A}_T^{\text{mr}}(R).$$
\end{prop}

\begin{proof}
Let $M$ be a Matlis reflexive $R$-module.  We show that
$M$ satisfies the defining conditions of $\mathcal{G}_{\md{T}}^{\text{mr}}(R)$ if
and only if $M$ satisfies the defining conditions of $\mathcal{A}_T^{\text{mr}}(R)$.
For the isomorphisms, consider the following commutative diagram
$$\xymatrix{
M\ar[r]^-{\delta_M^{\md{T}}}\ar[d]^-{\gamma_M^T}
&\hom{\hom{M}{\md{T}}}{\md{T}}\ar[d]^{\cong}\\
\hom{T}{\Otimes TM}\ar[d]^{\hom{T}{\delta_{\Otimes TM}^E}}_{\cong}
&\hom{\Otimes{\hom{M}{\md{T}}}{T}}{E}\ar[d]^{\cong}\\
\hom{T}{\hom{\md{(\Otimes TM)}}{E}}\ar[r]_-{\cong}
&\hom{T}{\hom{\hom{M}{\md{T}}}{E}}.
}$$
The unspecified isomorphisms are Hom-tensor adjointness.
The module $\Otimes TM$ is artinian by~\cite[Lemma 1.19 and Theorem 3.1]{kubik:hamm1}.
Fact~\ref{fact110120b} implies that the
map $\delta_{\Otimes TM}^E$, and hence
the map $\hom{T}{\delta_{\Otimes TM}^E}$, is an isomorphism.
Therefore the map $\gamma_M^T$ is an isomorphism if and only
if the map $\delta_M^{\md{T}}$ is an isomorphism.

Next we show that for all $i>0$ we have
$\Ext{i}{M}{\md{T}}=0$ if and only if $\Tor{i}{M}{T}=0$.
By~\cite[Remark 1.9]{kubik:hamm1}, we have 
$\Ext{i}{M}{\md{T}}\cong\md{\Tor{i}{M}{T}}$.
Because the Matlis dual of a module is zero if and only
if the module is zero, we conclude that
$\Ext{i}{M}{\md{T}}=0$ if and only if $\Tor{i}{M}{T}=0$ for all $i>0$.

Next we show that for all $i>0$ we have 
$\Ext{i}{\hom{M}{\md{T}}}{\md{T}}=0$ if and only if
$\Ext{i}{T}{\Otimes MT}=0$.
Hom-tensor adjointness explains the first step
in the following sequence
\begin{align*}
\Ext{i}{\hom{M}{\md{T}}}{\md{T}}&\cong
\Ext{i}{\md{(\Otimes MT)}}{\md{T}}\\
&\cong\Ext{i}{T^{\vee\vee}}{(\Otimes MT)^{\vee\vee}}\\
&\cong\Ext{i}{T}{\Otimes MT}.
\end{align*}
The second step follows from Lemma~\ref{lem110120}
and the fact that $T$ is artinian and thus Matlis reflexive.
The third step follows from the fact that $T$ and $\Otimes MT$ are
artinian and hence Matlis reflexive; see~\cite[Corollary 3.9]{kubik:hamm1}.
\end{proof}

\begin{cor}\label{cor120203}
Assume that  $R$ is complete and let $T$ be a quasidualizing $R$-module.
Then 
%
 $\mathcal{G}_{\md{T}}^{\text{noeth}}(R)=\mathcal{A}_T^{\text{noeth}}(R)$ and
$\mathcal{G}_{\md{T}}^{\text{artin}}(R)=\mathcal{A}_T^{\text{artin}}(R)$.
\end{cor}

\begin{prop}\label{prop120116}
If $R$ is complete and $T$ is a quasidualizing $R$-module,
then 
$$\mathcal{G}_{T}^{\text{mr}}(R)=\mathcal{A}_{\md{T}}^{\text{mr}}(R).$$
\end{prop}

\begin{proof}
Let $M$ be a Matlis reflexive $R$-module.
We show that $M$ satisfies the defining conditions of $\mathcal{G}_T^{\text{mr}}(R)$
if and only if $M$ satisfies the defining conditions of $\mathcal{A}_{\md{T}}^{\text{mr}}(R)$.
For the isomorphisms, consider the following commutative diagram
$$\xymatrix{
M
\ar[rrr]^-{\delta_M^T}
\ar[d]^{\gamma_M^{\md{T}}}
&&&\hom{\hom MT}{T}
\ar[dddd]^{\cong}_{\hom{\hom MT}{\delta_T^E}}\\
\hom{\md{T}}{\Otimes{\md{T}}{M}}\ar[d]^{\hom{\md{T}}{\delta^E_{\md{T}\otimes M}}}
&&&\\
\hom{\md{T}}{\mdd{(\Otimes{\md{T}}{M})}}\ar[d]^{\cong}&&&\\
\hom{\hom{\Otimes{\md{T}}{M}}{E}}{\mdd{T}}\ar[d]^{\cong}&&&\\
\hom{\hom{M}{\mdd{T}}}{\mdd{T}}\ar[rrr]^-{\cong}_{\hom{\hom{M}{\delta^E_T}}{\mdd{T}}}&&&
\hom{\hom MT}{\mdd{T}}
}$$
where the unlabeled isomorphisms are Hom-tensor adjointness and Hom-swap.  Since
$T$ is artinian and hence Matlis reflexive, both the right hand map and the bottom
map are isomorphisms.  
The module $\Otimes{\md{T}}{M}$ is Matlis reflexive by~\cite[Corollary~3.6]{kubik:hamm1}.
Thus the map $\delta^E_{\md{T}\otimes M}$ and hence the map
$\hom{\md{T}}{\delta^E_{\md{T}\otimes M}}$ is an  isomorphism.
Therefore the map $\gamma^{\md{T}}_{M}$ is an isomorphism
if and only if the map $\delta^T_M$ is an isomorphism.

Next we show that for all $i>0$ we have $\Ext{i}{M}{T}=0$ if and only if
$\Tor{i}{\md{T}}{M}=0$.  The fact that $T$ is artinian
and hence Matlis reflexive explains the first step in the
following sequence
$$\Ext{i}{M}{T}\cong\Ext{i}{M}{\mdd{T}}\cong
\md{\Tor{i}{M}{\md{T}}}\cong\md{\Tor{i}{\md{T}}{M}}.$$
The second step follows from~\cite[Remark~1.9]{kubik:hamm1}
and the last step follows from the commutativity  of the tensor product.
Because the Matlis dual of a module is zero if and only if the module is zero, we conclude that 
$\Ext{i}{M}{T}=0$ if and only if $\Tor{i}{\md{T}}{M}=0$ for all $i>0$.

Next we show that for all $i>0$ we have
$\Ext{i}{\hom MT}{T}=0$ if and only if $\Ext{i}{\md{T}}{\Otimes{\md{T}}{M}}=0$.
The fact that $T$ is artinian and
hence Matlis reflexive explains the first and third steps in the following
sequence
\begin{align*}
\Ext{i}{\hom MT}{T}&\cong
\Ext{i}{\hom{M}{\mdd{T}}}{T}\\
&\cong\Ext{i}{\hom{\Otimes{M}{\md{T}}}{E}}{T}\\
&\cong\Ext{i}{\hom{\Otimes{M}{\md{T}}}{E}}{\mdd{T}}\\
&\cong\Ext{i}{\md{T}}{\Otimes{M}{\md{T}}}.
\end{align*}
The second step follows from Hom-tensor adjointness and the last step
follows from Lemma~\ref{lem110120}.
\end{proof}

\begin{cor}\label{cor110225c}
Assume that  $R$ is complete and let $T$ be a quasidualizing $R$-module.
Then 
%
$\mathcal{G}_T^{\text{noeth}}(R)=\mathcal{A}_{\md{T}}^{\text{noeth}}(R)$ and
$\mathcal{G}_T^{\text{artin}}(R)=\mathcal{A}_{\md{T}}^{\text{artin}}(R)$.
\end{cor}

The above results show that the classes $\mathcal{G}_T^{\text{mr}}(R)$,
$\mathcal{G}_T^{\text{artin}}(R)$, and $\mathcal{G}_T^{\text{noeth}}(R)$
do not exhibit some of the same properties as the class
$\mathcal{G}_C^{\text{noeth}}(R)$, where $C$ is semidualizing.  
For instance, we consider
the following property.  We say a class of $R$-modules $\mathcal{C}$
satisfies the two-of-three condition  if given an exact
sequence of $R$-module homomorphisms
$0\rightarrow L_1\rightarrow L_2\rightarrow L_3\rightarrow 0$,
 when any two of the modules are in $\mathcal{C}$, so is the third.
The two-of-three condition
holds for some classes of modules and not for others.
For example, the class of noetherian modules and the class of artinian
modules both satisfy the two-of-three condition.
On the other hand, the class $\mathcal{G}_{C}^{\text{noeth}}(R)$
does not satisfy the two-of-three condition when $C$ is semidualizing.
In contrast, the next result shows that the class $\mathcal{G}_T^{full}(R)$
satisfies the two-of-three condition when the ring is complete.
This is somewhat surprising since the definitions of 
$\mathcal{G}_{C}^{\text{noeth}}(R)$ and $\mathcal{G}_{T}^{full}(R)$
are so similar.   First we need a lemma.
In the language of~\cite{holm:fear}, is says that
quasidualizing implies faithfully quasidualizing.

\begin{lem}\label{lem110125}
Let $L$ and $T$ be $R$-modules
such that $T$ is quasidualizing.  If one has $\hom LT=0$, then $L=0$.
\end{lem}

\begin{proof}
Assume that $\hom LT=0$.

Case 1: $T=E$.   Because $\hom LE=0$, we have
$L^{\vee\vee}=0$.  Since the map $\delta_L^E
:L\rightarrow L^{\vee\vee}$ is injective by Fact~\ref{fact110203}, 
we conclude that $L=0$.

Case 2: $R$ is complete.  
Then $T$ is Matlis reflexive and we have $0=\hom LT\cong\hom{\md{T}}{\md{L}}$ from Lemma~\ref{lem110120}.
Since $\md{T}$ is semidualizing by Proposition~\ref{prop110113},
we have $\md{L}=0$ by~\cite[Proposition 3.6]{holm:fear}.
By Case 1, we conclude that $L=0$.

Case 3: the general case.
The first step in the following sequence is by assumption
$$0=\hom LT
\cong\hom {L}{\hom[\comp R]{\comp R}{T}}
\cong\hom[\comp R]{\Otimes{\comp R}{L}}{T}.$$
The second step follows from the fact that $T$ is artinian and hence
has an $\comp R$ structure and the third step is from Hom-tensor adjointness.
Since $T$ is a quasidualizing $\comp R$-module, we can apply Case 2 to conclude
that $\Otimes{\comp R}{L}=0$.  Then $L=0$ because $\comp R$
is faithfully flat over $R$.
\end{proof}

\begin{question}
Does a version of  Lemma~\ref{lem110125} hold for $\Otimes{T}{-}$
as in~\cite{holm:fear}?
\end{question}

\begin{thm}\label{prop110125}
Assume that $R$ is complete and let $T$ be a quasidualizing
$R$-module.  Then $\mathcal{G}_T^{full}(R)$ satisfies the two-of-three condition.
\end{thm}

\begin{proof}
Let \begin{equation}\label{eqprop110125}
0\rightarrow L_1\xrightarrow{f} L_2\xrightarrow{g} L_3\rightarrow 0
\end{equation}
be an exact sequence of $R$-module homomorphisms
and let $(-)^T=\hom -T$.  
There are two conditions to check and three cases.
We will deal with the case when $L_1, L_2\in\mathcal{G}_T^{full}(R)$.
The case where $L_2, L_3\in\mathcal{G}_T^{full}(R)$
is similar. The case where $L_1, L_3\in\mathcal{G}_T^{full}(R)$ is also
similar but easier.

Assume that $L_1, L_2\in\mathcal{G}_T^{full}(R)$.
Then we have $\Ext{i}{L_1}{T}=0=\Ext{i}{L_2}{T}$
for all $i>0$.
The following portion of 
the long exact sequence in $\Ext{i}{-}{T}$ associated
to the short exact sequence~\eqref{eqprop110125} 
\begin{equation}\label{eqprop110125a}
\cdots\rightarrow\Ext{i-1}{L_1}{T}
\rightarrow\Ext{i}{L_3}{T}\rightarrow
\Ext{i}{L_2}{T}\rightarrow
\Ext{i}{L_1}{T}\rightarrow\cdots
\end{equation}
shows that $\Ext{i}{L_3}{T}=0$ for all $i> 1$.
For the case where $i=1$, we apply $(-)^T$ 
to the following portion of the long exact sequence
$$0\rightarrow(L_3)^T\rightarrow(L_2)^T\rightarrow
(L_1)^T\rightarrow\Ext{1}{L_3}{T}\rightarrow 0$$
to obtain exactness in the top row of the following commutative diagram
%
$$\xymatrix {
0\ar[r]
&(\Ext{1}{L_3}{T})^T\ar[r]
&(L_1)^{TT}\ar[r]^{ f^{TT}}
&(L_2)^{TT}\\
&0\ar[r]
&L_1\ar[r]^{f}\ar[u]^{\cong}_{\delta_{L1}^T}
&L_2\ar[u]^{\cong}_{\delta_{L_2}^T}.
}$$
Since $f$ is an injective map, the diagram shows that  $f^{TT}$ is
an injective map.  Hence we have  $(\Ext{1}{L_3}{T})^T=0$.
From Lemma~\ref{lem110125} we conclude that $\Ext{1}{L_3}{T}=0$.

Next we show that $\Ext{i}{\hom{L_3}{T}}{T}=0$ for all $i> 0$.
From the argument above we have the exact sequence
\begin{equation}\label{eqprop110125b}
0\rightarrow(L_3)^T
\rightarrow(L_2)^T
\rightarrow(L_1)^T\rightarrow 0.
\end{equation}
In a similar, but easier,  manner than above, the long exact sequence in
$\Ext{i}{-}{T}$ shows that if $L_1, L_2\in\mathcal{G}_T^{full}(R)$, then
$\Ext{i}{\hom{L_3}{T}}{T}=0$ for all $i>0$.

Lastly, we show that the map $\delta_{L_3}^T$ is an isomorphism.
From the 
short exact sequence~\eqref{eqprop110125} and
as a consequence of the above argument together with the
short exact sequence~\eqref{eqprop110125b},
we obtain the following commutative diagram
with exact rows
$$\xymatrix{
0\ar[r]
&L_1\ar[r]^{f}\ar[d]^{\delta_{L_1}^T}_{\cong}
&L_2\ar[r]^{g}\ar[d]^{\delta_{L_2}^T}_{\cong}
&L_3\ar[r]\ar[d]^{\delta_{L_3}^T}
&0\\
0\ar[r]
&(L_1)^{TT}\ar[r]^-{f^{TT}}
&(L_2)^{TT}\ar[r]^-{g^{TT}}
&(L_3)^{TT}\ar[r]
&0.
}$$
Since $L_1,L_2$ are in $\mathcal{G}_T^{full}(R)$, 
the maps $\delta_{L_1}^T$ and $\delta_{L_2}^T$ are
isomorphisms.  By the Snake Lemma,
we conclude that $\delta_{L_3}^T$ is an isomorphism.
\end{proof}

\begin{cor}\label{cor120111}
Assume that $R$ is complete and let $T$ be a quasidualizing
$R$-module.  Then $\mathcal{G}_T^{\text{artin}}(R)=\mathcal{A}_{\md{T}}^{\text{artin}}(R)$,
$\mathcal{G}_T^{\text{noeth}}(R)=\mathcal{A}_{\md{T}}^{\text{noeth}}(R)$,
and $\mathcal{G}_T^{mr}(R)$
satisfy the two-of-three condition.
\end{cor}

\begin{proof}
Apply Theorem~\ref{prop110125} and Corollary~\ref{cor110225c}.
\end{proof}


\begin{thebibliography}{1}

\bibitem{atiyah:ica}
M.~F. Atiyah and I.~G. Macdonald, \emph{Introduction to commutative algebra},
  Addison-Wesley Publishing Co., Reading, Mass.-London-Don Mills, Ont., 1969.
  \MR{0242802 (39 \#4129)}

\bibitem{enochs:rha}
E.~E. Enochs and O.~M.~G. Jenda, \emph{Relative homological algebra}, de
  Gruyter Expositions in Mathematics, vol.~30, Walter de Gruyter \& Co.,
  Berlin, 2000. \MR{1753146 (2001h:16013)}

\bibitem{holm:fear}
H.\ Holm and D.\ White, \emph{Foxby equivalence over associative rings}, J.
  Math. Kyoto Univ. \textbf{47} (2007), no.~4, 781--808. \MR{2413065}

\bibitem{ishikawa:oimafm}
T.\ Ishikawa, \emph{On injective modules and flat modules}, J. Math. Soc. Japan
  \textbf{17} (1965), 291--296. \MR{0188272 (32 \#5711)}

\bibitem{kubik:hamm1}
B.~Kubik, M.~J. Leamer, and S.~Sather-Wagstaff, \emph{Homology of artinian and
  mini-max modules, {I}}, J. Pure Appl. Algebra \textbf{215} (2011), no.~10,
  2486--2503.

\bibitem{matsumura:crt}
H.\ Matsumura, \emph{Commutative ring theory}, second ed., Studies in Advanced
  Mathematics, vol.~8, University Press, Cambridge, 1989. \MR{90i:13001}

\bibitem{rotman:iha2}
Joseph~J. Rotman, \emph{An introduction to homological algebra}, second ed.,
  Universitext, Springer, New York, 2009. \MR{2455920 (2009i:18011)}

\bibitem{vasconcelos:dtmc}
W.~V. Vasconcelos, \emph{Divisor theory in module categories}, North-Holland
  Publishing Co., Amsterdam, 1974, North-Holland Mathematics Studies, No. 14,
  Notas de Matem\'atica No. 53. [Notes on Mathematics, No. 53]. \MR{0498530 (58
  \#16637)}

\end{thebibliography}
\providecommand{\bysame}{\leavevmode\hbox to3em{\hrulefill}\thinspace}
\providecommand{\MR}{\relax\ifhmode\unskip\space\fi MR }
\providecommand{\MRhref}[2]{%
  \href{http://www.ams.org/mathscinet-getitem?mr=#1}{#2}
}
\providecommand{\href}[2]{#2}

\end{document}